\documentclass{amsart}
\usepackage{amssymb}
\usepackage{amsmath}
\usepackage{mathrsfs}
\usepackage{bbm}
\usepackage{stmaryrd}
\usepackage{tikz}
\usetikzlibrary{cd}
\usetikzlibrary{decorations.markings}

\tikzset{
  dot/.style={circle,fill,draw,inner sep=0mm,minimum size=1mm},        
  hdot/.style={circle,fill=white,draw,inner sep=0mm,minimum size=1mm}, 
}

\renewcommand{\hat}{\widehat}

\newcommand{\Z}{\mathbb{Z}}
\newcommand{\bk}{\Bbbk}

\DeclareMathOperator{\Sym}{Sym}
\newcommand{\id}{\mathrm{id}}
\newcommand{\simto}{\overset{\sim}{\to}}
\newcommand{\la}{\langle}
\newcommand{\ra}{\rangle}

\newcommand{\Kb}{K^{\mathrm{b}}}
\newcommand{\Db}{D^{\mathrm{b}}}

\newcommand{\uv}{{\underline{v}}}
\newcommand{\uw}{{\underline{w}}}

\newcommand{\Diag}{\mathscr{D}}
\newcommand{\ustar}{\mathbin{\underline{\star}}}

\newcommand{\Tilt}{\mathrm{Tilt}}
\newcommand{\Tmon}{\widetilde{\mathcal{T}}}
\newcommand{\hatstar}{\mathbin{\widehat{\star}}}

\newcommand{\FM}{\mathsf{FM}}

\newcommand{\cF}{\mathcal{F}}
\newcommand{\cG}{\mathcal{G}}

\DeclareMathOperator{\Hom}{Hom}
\DeclareMathOperator{\End}{End}
\DeclareMathOperator{\Ext}{Ext}
\DeclareMathOperator{\uHom}{\underline{Hom}}
\DeclareMathOperator{\uEnd}{\underline{End}}
\DeclareMathOperator{\gHom}{\mathbb{H}\mathsf{om}}
\DeclareMathOperator{\gEnd}{\mathbb{E}\mathsf{nd}}

\newcommand{\ngmod}{\mathrm{mod}}
\newcommand{\gmod}{\mathrm{gmod}}
\newcommand{\lh}{\text{-}}

\renewcommand{\ll}{\llbracket}
\newcommand{\rr}{\rrbracket}
\newcommand{\hoch}{{\mathrm{Ext}}}
\newcommand{\HcF}{\cF^\hoch}

\newcommand{\GL}{\mathrm{GL}}
\newcommand{\bim}{{\mathrm{bim}}}

\renewcommand{\k}{\Bbbk}
\renewcommand{\a}{\alpha}
\renewcommand{\b}{\beta}
\renewcommand{\d}{\delta}
\newcommand{\sqmatrix}[1]{\left[\begin{matrix} #1\end{matrix}\right]}
\newcommand{\smMatrix}[1]{\left[\begin{smallmatrix}#1\end{smallmatrix}\right]}

\newcommand{\one}{\mathbbm{1}}
\newcommand{\ip}[1]{\langle #1\rangle}
\newcommand{\CS}{\mathscr{C}}
\newcommand{\DS}{\mathscr{D}}
\newcommand{\Seq}{\operatorname{Seq}}
\newcommand{\Seqb}{\Seq^{\mathrm{b}}}
\newcommand{\Ch}{\operatorname{Ch}}
\newcommand{\Chb}{\Ch^{\mathrm{b}}}

\newcommand{\BS}{\mathrm{BS}}
\newcommand{\HBS}{\mathrm{BS}^\hoch}
\newcommand{\HDiag}{\Diag^\hoch}
\newcommand{\HBE}{{\mathsf{BE}^\hoch}}
\newcommand{\HBEdg}{\mathsf{BE}^{\hoch,\mathrm{dg}}}
\newcommand{\HFM}{{\mathsf{FM}^\hoch}}
\newcommand{\HFMdg}{\mathsf{FM}^{\hoch,\mathrm{dg}}}

\newcommand{\HLambdav}{\Lambda_\hoch^\vee}
\newcommand{\HR}{R_\hoch}

\newcommand{\hphi}{{\hat{\phi}}}
\newcommand{\hheta}{{\hat{\eta}^\hoch}}
\newcommand{\hheps}{{\hat{\epsilon}^\hoch}}

\newcommand{\sym}{\mathrm{sd}}
\newcommand{\HPhi}{\Phi_\sym^\hoch}

\newsavebox\linemor
\savebox\linemor{%
\begin{tikzpicture}[scale=0.5,thick,baseline=-2pt]
 \draw (0,-0.5) to (0,0.5);
\end{tikzpicture}%
}
\newsavebox\downdot
\savebox\downdot{%
\begin{tikzpicture}[scale=0.3,thick,baseline=-5pt]
 \draw (0,-0.5) to (0,0.5); \node[dot] at (0,-0.5) {};
\end{tikzpicture}%
}
\newsavebox\updot
\savebox\updot{%
\begin{tikzpicture}[scale=0.3,thick,baseline]
 \draw (0,-0.5) to (0,0.5); \node[dot] at (0,0.5) {};
\end{tikzpicture}%
}
\newsavebox\updotdowndot
\savebox\updotdowndot{%
\begin{tikzpicture}[scale=0.3,thick,baseline=-2pt]
 \draw (0,-1) to (0,-0.4); \node[dot] at (0,-0.4) {};
 \draw (0,0.4) to (0,1); \node[dot] at (0,0.4) {};
\end{tikzpicture}%
}
\newsavebox\barbell
\savebox\barbell{%
\begin{tikzpicture}[scale=0.3,thick,baseline]
 \draw (0,-0.5) to (0,0.5); \node[dot] at (0,-0.5) {}; \node[dot] at (0,0.5) {};
\end{tikzpicture}%
}

\newsavebox\invymor
\savebox\invymor{%
\begin{tikzpicture}[yscale=0.2,xscale=0.1,baseline,thick] \draw (-1,-1) -- (0,0) -- (1,-1); \draw (0,0) -- (0,1); \end{tikzpicture}%
}
\newsavebox\ymor
\savebox\ymor{%
\begin{tikzpicture}[yscale=-0.2,xscale=0.1,baseline,thick] \draw (-1,-1) -- (0,0) -- (1,-1); \draw (0,0) -- (0,1); \end{tikzpicture}%
}

\newsavebox\hdot
\savebox\hdot{%
\begin{tikzpicture}[scale=0.5,thick,baseline=-2pt]
 \draw (0,-0.5) to (0,0.5); \node[hdot] at (0,0) {};
\end{tikzpicture}%
}
\newsavebox\hdowndot
\savebox\hdowndot{%
\begin{tikzpicture}[scale=0.3,thick,baseline=-5pt]
 \draw (0,-0.5) to (0,0.5); \node[hdot] at (0,-0.5) {};
\end{tikzpicture}%
}
\newsavebox\hupdot
\savebox\hupdot{%
\begin{tikzpicture}[scale=0.3,thick,baseline]
 \draw (0,-0.5) to (0,0.5); \node[hdot] at (0,0.5) {};
\end{tikzpicture}%
}
\newsavebox\hupdotdowndot
\savebox\hupdotdowndot{%
\begin{tikzpicture}[scale=0.3,thick,baseline=-2pt]
 \draw (0,-1) to (0,-0.4); \node[hdot] at (0,-0.4) {};
 \draw (0,0.4) to (0,1); \node[dot] at (0,0.4) {};
\end{tikzpicture}%
}
\newsavebox\updothdowndot
\savebox\updothdowndot{%
\begin{tikzpicture}[scale=0.3,thick,baseline=-2pt]
 \draw (0,-1) to (0,-0.4); \node[dot] at (0,-0.4) {};
 \draw (0,0.4) to (0,1); \node[hdot] at (0,0.4) {};
\end{tikzpicture}%
}
\newsavebox\hbarbell
\savebox\hbarbell{%
\begin{tikzpicture}[scale=0.3,thick,baseline]
 \draw (0,-0.5) to (0,0.5); \node[dot] at (0,-0.5) {}; \node[hdot] at (0,0) {}; \node[dot] at (0,0.5) {};
\end{tikzpicture}%
}
\newsavebox\hobarbell
\savebox\hobarbell{%
\begin{tikzpicture}[scale=0.3,thick,baseline=-3pt]
 \draw (0,-0.5) to (0,0.5); \node[hdot] at (0,-0.5) {}; \node[dot] at (0,0.5) {};
\end{tikzpicture}%
}
\newsavebox\ohbarbell
\savebox\ohbarbell{%
\begin{tikzpicture}[scale=0.3,thick,baseline=-3pt]
 \draw (0,-0.5) to (0,0.5); \node[dot] at (0,-0.5) {}; \node[hdot] at (0,0.5) {};
\end{tikzpicture}%
}
\newsavebox\ohdowndot
\savebox\ohdowndot{%
\begin{tikzpicture}[scale=0.5,thick,baseline=-3pt]
 \draw (0,-0.5) to (0,0.5); \node[dot] at (0,-0.5) {}; \node[hdot] at (0,0) {};
\end{tikzpicture}%
}
\newsavebox\ohupdot
\savebox\ohupdot{%
\begin{tikzpicture}[scale=-0.5,thick,baseline=-3pt]
 \draw (0,-0.5) to (0,0.5); \node[dot] at (0,-0.5) {}; \node[hdot] at (0,0) {};
\end{tikzpicture}%
}

\numberwithin{equation}{section}
\numberwithin{figure}{section}
\newtheorem{thm}{Theorem}[section]
\newtheorem{lem}[thm]{Lemma}

\newtheorem{conj}[thm]{Conjecture}
\theoremstyle{definition}
\newtheorem{defn}[thm]{Definition}

\theoremstyle{remark}
\newtheorem{rmk}[thm]{Remark}
\newtheorem{ex}[thm]{Example}

\title{Ext-enhanced monoidal Koszul duality for $\GL_2$}

\author{Matthew Hogancamp}
\address{Department of Mathematics, University of Southern California, Los Angeles, CA 90089, U.S.A.}
\email{hogancam@usc.edu}

\author{Shotaro Makisumi}
\address{Department of Mathematics, Columbia University, New York, NY 10027, U.S.A.}
\email{makisumi@math.columbia.edu}

\date{September 12, 2019}

\begin{document}

\begin{abstract}
The Hecke category participates in an equivalence called monoidal Koszul duality, which exchanges it with the category of (Langlands-dual) ``free-monodromic tilting sheaves.'' Motivated by a recent conjecture of Gorsky and the first-named author on HOMFLYPT link homology, we propose to enhance this duality with an additional grading. We provide evidence for this enhancement in the case of $\GL_2$, working in the language of the second-named author's joint work with Achar, Riche, and Williamson.
\end{abstract}

\maketitle

\section{Introduction}

In the recent preprint \cite{gh}, E.~Gorsky and the first-named author introduced ``$y$-ified $\GL_n$ Soergel bimodules,'' a certain deformation of complexes of $\GL_n$ Soergel bimodules, and used them to define a deformation of triply-graded link homology that they call ``$y$-ified homology.'' They conjectured that $y$-ified homology restores a missing $q$-$t$ symmetry in the triply-graded HOMFLYPT link homology of Khovanov--Rozansky \cite{kr,kho}, and that this symmetry comes from a monoidal triangulated autoequivalence of the category of $y$-ified $\GL_n$ Soergel bimodules.

Soergel bimodules are an algebraic incarnation of the Hecke category, a monoidal category which plays a central role in geometric representation theory. The Hecke category participates in a monoidal triangulated equivalence known as monoidal Koszul duality, which in characteristic $0$ is due to Bezrukavnikov--Yun \cite{by} and exchanges the Hecke category associated to a reductive group $G$ with ``free-monodromic tilting sheaves'' associated to the Langlands dual group $G^\vee$.

Monoidal Koszul duality involves bigraded categories. On the other hand, HOMFLYPT homology has a third ``Hochschild'' grading (in Khovanov's construction) coming from Ext groups between Soergel bimodules computed in the abelian category of graded bimodules. The aim of this paper is to suggest that monoidal Koszul duality should similarly admit an Ext-enhancement to an equivalence of triply-graded categories.

In recent work \cite{amrwa,amrwb}, P.~N.~Achar, S.~Riche, G.~Williamson, and the second-named author proposed a new construction of free-monodromic tilting sheaves that also makes sense in positive characteristic, then used this to establish a positive characteristic monoidal Koszul duality. In this paper, we work in the language of \cite{amrwa,amrwb} to provide some evidence for an Ext-enhanced monoidal Koszul duality for $\GL_2$. In particular, we introduce an Ext-enhancement of the $\GL_2$ diagrammatic Hecke category of Elias--Khovanov \cite{ek}, then start with this category to define similarly Ext-enhanced versions of the categories of \cite{amrwa}. Assuming that the exchange law of a monoidal category continues to hold for Ext-enhanced free-monodromic tilting sheaves, our main result (Theorem~\ref{thm:hmkd}) constructs an Ext-enhanced monoidal Koszul duality functor.

\subsection{Acknowledgements}
The first-named author was supported by NSF grant DMS-1702274. Initial computations for this work were done while the second-named author was in residence at the Mathematical Sciences Research Institute during Spring 2018, supported by NSF grant DMS-1440140.

\section{Preliminaries} \label{s:prelim}

\subsection{Graded categories}
\label{ss:gradedcats}
In this paper we consider $\Z^2$- and $\Z^3$-graded monoidal categories.  Let $\k$ be a commutative ring, and let $\CS$ be a $\k$-linear category.  Let $\Gamma$ be an abelian group which acts strictly on $\CS$.  That is to say, for each $\gamma\in \Gamma$ there is an autoequivalence $\Sigma_\gamma:\CS\rightarrow \CS$, and these satisfy the relations
\[
\Sigma_{\gamma}\circ \Sigma_{\gamma'} = \Sigma_{\gamma+\gamma'},\qquad\qquad \Sigma_0 = \id_{\CS}.
\]

We let $\Hom^{\Gamma}(X,Y) :=\bigoplus_{\gamma\in \Gamma}\Hom_{\CS}(\Sigma_\gamma(X),Y)$ denote the $\Gamma$-graded $\k$-module of homs.  We let $\operatorname{en}_\Gamma(\CS)$ denote the category with the same objects as $\CS$, but with hom spaces $\Hom_{\CS}^\Gamma$.

This defines a 2-functor from categories with a strict $\Gamma$-action to categories enriched in $\Gamma$-graded $\k$-modules.

\begin{rmk}
Sometimes it is conventional to let $(\gamma)$ to denote the \emph{downward} shift by $\gamma$, i.e.~$X(\gamma):=\Sigma_{-\gamma}(X)$.  There is a canonical isomorphism
\[
\Hom_\CS^\Gamma(X,Y) = \bigoplus_{\gamma\in \Gamma}(X,Y(\gamma)).
\]
\end{rmk}

There is a 2-functor the other direction.  Suppose $\DS$ is a category enriched in $\Gamma$-graded $\k$-modules.  Define $\operatorname{un}(\DS)$ to be the category whose objects are pairs $(X,\gamma)$ consisting of an object $X\in \DS$ and an element $\gamma\in \Gamma$.  The hom spaces in $\operatorname{un}(\DS)$ are
\[
\Hom_{\operatorname{un}(\DS)}((X,\gamma),(Y,\gamma')) :=\Hom_{\DS}^{\gamma-\gamma'}(X,Y).
\] 

\begin{rmk}
If $\DS= \operatorname{en}(\CS)$ for some category $\CS$ with a strict $\Gamma$ action, then the objects $(X,\gamma)$ and $(\Sigma_\gamma(X),0)$ are isomorphic in $\operatorname{un}(\DS)$, via the isomorphism
\begin{multline*}
\id_{\Sigma_\gamma(X)}\in \Hom_{\CS}(\Sigma_\gamma(X),\Sigma_\gamma(X)) \\=  \Hom_{\DS}^\gamma(X,\Sigma_\gamma(X)) = \Hom_{\operatorname{un}(\DS)}((X,\gamma),(\Sigma_\gamma(X),0)).
\end{multline*}
Thus the functor $\CS\rightarrow  \operatorname{un}(\operatorname{en}(\CS))$ sending $X\mapsto (X,0)$ is an equivalence of categories.
\end{rmk}

\begin{ex}
If $A$ is a $\Gamma$-graded $\k$-algebra, then we let $\CS:=A\lh\gmod$ denote the category of $\Gamma$-graded $A$-modules with degree zero morphisms.  We let $\Sigma_\gamma$ denote the grading shift functor; it acts on objects by $\Sigma_\gamma(X)_{\gamma'} :=X_{\gamma'-\gamma}$.
\end{ex}

\subsection{Graded and super-monoidal categories}
\label{ss:gradedmonoidalcats}
Suppose $\CS$ is a $\k$-linear monoidal category with a strict $\Gamma$-action.  We typically want the $\Gamma$-action and the monoidal structure to be compatible. This compatibility takes the form of a pair of natural isomorphisms
\[
\a_{\gamma,X}:\Sigma_\gamma(\one)\otimes X \buildrel \cong \over\rightarrow \Sigma_{\gamma}(X),\qquad\qquad \b_{\gamma,X}:\Sigma_{\gamma}(\one)\otimes X \buildrel\cong\over\rightarrow X\otimes \Sigma_{\gamma}(\one),
\]
for all $X\in \CS$, subject to certain coherence conditions (see \cite{hog} for more details).

Most importantly, we require that $\Sigma_\gamma(\one)$, equipped with the \emph{braiding morphisms} $\b_{\gamma,X}$, has the structure of an object in the Drinfeld center of $\CS$.  Furthermore, if $X=\Sigma_\gamma(\one)$, then $\b_{\gamma,X}$ equals $(-1)^{\ip{\gamma,\gamma}} \id_{\Sigma_\gamma(X)\otimes \Sigma_\gamma(X)}$. The sign here is determined by a symmetric bilinear pairing $\ip{\ ,\ }:\Gamma\times \Gamma\rightarrow \Z/2$, called the \emph{parity form}.

\begin{rmk}
Typically the parity form is determined by a group homomorphism $p:\Gamma\rightarrow \Z/2$ via $\ip{\gamma,\gamma'} = p(\gamma)p(\gamma')$. More generally, the parity form may also be determined by a group homomorphism $p:\Gamma\rightarrow (\Z/2)^r$ via
\[
\ip{\gamma,\gamma'} :=p(\gamma)\cdot p(\gamma') = p_1(\gamma)p_1(\gamma')+\cdots+p_r(\gamma)p_r(\gamma'),
\]
where $p(\gamma)=(p_1(\gamma),\ldots,p_r(\gamma))$.  
\end{rmk}

We will say that $\CS$ is equipped with a \emph{strict monoidal $\Gamma$-action} or that $\CS$ has the structure of a \emph{$\Gamma$-monoidal category} if it is equipped with a strict $\Gamma$-action and natural transformations $\a$, $\b$ as above.

\begin{ex}
Let $\CS=\Ch(\k\lh\ngmod)$ be the category of $\Z$-graded complexes of $\k$-modules with differentials of degree $1$.  Then $\CS$ is $\Z$-monoidal with tensor product $\otimes_\k$, extended to complexes via the usual sign rule, and the parity form is determined by the nontrivial homomorphism $\Z\twoheadrightarrow \Z/2$.
\end{ex}

Now, if $\CS$ has the structure of a $\Gamma$-monoidal category, then we can consider the enriched category $\operatorname{en}(\CS)$ defined in the previous section.  This category inherits a tensor product from $\CS$, but the tensor product of morphisms satisfies the graded exchange law
\[
(f\otimes g)\circ (f'\otimes g') = (-1)^{\ip{\deg g,\deg f'}} (f\circ f')\otimes (g\circ g').
\]

The enriched category $\operatorname{en}(\CS)$ is not, strictly speaking, a monoidal category. Rather, it is a \emph{super-monoidal category}, because of the above exchange law.

\begin{rmk}
Traditionally, the prefex ``super'' indicates $\Z/2$-graded categories.  We prefer a less restricted use, which instead refers to arbitrary $\Gamma$-graded categories equipped with a parity form $\Gamma\times \Gamma \rightarrow \Z/2$.
\end{rmk}

\begin{rmk}
The graded ring of endomorphisms $\End^\Gamma_{\CS}(\one)$ in a $\Gamma$-graded super-mon-oidal category $\CS$ is super-commutative with respect to the given parity form $\ip{\ ,\ }$.
\end{rmk}

\begin{rmk}
There are two natural ways to set up a graphical calculus for graded monoidal categories.  In the first, one considers only degree zero morphisms.  In this setup, the calculus works in exactly the usual way.  The central objects $\Sigma_\gamma(\one)$ can be used to encode morphisms of nonzero degree, and all the various signs are captured entirely by the signs involved in braiding $\Sigma_\gamma(\one)$ past itself.

In the second setup, one allows morphisms of arbitrary degree.  In this case, exchanging the heights of two distant morphisms introduces a sign given by the parity form. This setup will be referred to as the super-calculus for graded monoidal categories.  The super-calculus is more compact, but the signs are conventional and occasionally mysterious.  Both calculi are preferable in various instances.
\end{rmk}

\section{Ext-enhanced Bott--Samelson bimodules}
\label{sec:hsbim}

In the rest of this paper, we fix a field $\bk$ and let
\[
 V = \bk y_1 \oplus \bk y_2, \qquad V^* = \bk x_1 \oplus \bk x_2, \qquad \alpha_s = x_1 - x_2, \qquad \alpha_s^\vee = y_1 - y_2,
\]
where $(y_1, y_2)$ and $(x_1, x_2)$ are dual bases of the $\bk$-vector spaces $V$ and $V^*$. This is the $\GL_2$ realization (in the sense of Elias--Williamson \cite[\S3.1]{ew}) of the type $A_1$ Coxeter system $S_2 = \{\id, s\}$, acting on $V$ via $s(y_1) = y_2$ and $s(y_2) = y_1$.

Let us recall the associated monoidal category of Bott--Samelson bimodules. Let $R = \Sym(V^*)$, viewed as a $\Z$-graded $\bk$-algebra with $\deg V^* = 2$. Consider the $\Z$-graded $R$-bimodule $B^\bim_s = R \otimes_{R^s} R(1)$, where $R^s \subset R$ denotes the $s$-invariants, and $(1)$ shifts the bimodule degree down by $1$. Given an expression $\uw = (s, \ldots, s)$, define the Bott--Samelson bimodule
\[
 B^\bim_\uw = B^\bim_s \otimes_R \cdots \otimes_R B^\bim_s.
\]
For the empty word, $B^\bim_\varnothing = R$. Let $R\lh\gmod\lh R$ be the category of $\Z$-graded $R$-bimodules and bimodule homomorphisms of degree $0$. Then $\BS_2$ is defined to be the full subcategory of $R\lh\gmod\lh R$ whose objects are $B^\bim_\uw(j)$ for expressions $\uw$ and $j \in \Z$.

Now, consider the bounded derived category $\Db(R\lh\gmod\lh R)$. We denote the cohomological shift by $\lceil1\rfloor$ (since $[1]$ will be reserved for another shift later on).
\begin{defn}
 The category of \emph{Ext-enhanced $\GL_2$ Bott--Samelson bimodules} $\HBS_2$ is the smallest full subcategory of $\Db(R\lh\gmod\lh R)$ containing $\BS_2$ (viewed as complexes supported in cohomological degree $0$) and closed under cohomological shift.
\end{defn}
In other words, objects of $\HBS_2$ are of the form $B(m)\lceil n \rfloor$, where $B$ is a Bott--Samelson bimodule and $m,n \in \Z$. For two such objects,
\[
 \Hom_{\HBS_2}(B(m)\lceil n \rfloor, B'(m')\lceil n' \rfloor) = \Ext^{n'-n}_{R\lh\gmod\lh R}(B(m),B'(m')).
\]
Instead of the cohomological shift, we will work primarily with the combined shift
\[
 \ll1\rr := (-2)\lceil 1 \rfloor.
\]
The category $\HBS_2$ is thus equipped with two grading shifts $(1), \ll1\rr$. Given two objects $B, B' \in \HBS_2$, define the bigraded $\bk$-module $\gHom(B,B')$ by
\[
 \gHom^{j,k}(B, B') = \Hom_{\HBS_2}(B, B'(j)\ll k \rr).
\]

Let us discuss the monoidal structure on these categories. The tensor product $\otimes_R$ makes $\BS_2$ into a monoidal category with a grading shift $(1)$. The monoidal structure on $\Db(R\lh\gmod\lh R)$ is defined by the derived tensor product over $R$.  However, since each bimodule in $\BS_2$ is free as a left and right $R$-module, the derived tensor product of such bimodules coincides with the usual tensor product, and the monoidal structure is defined by the ordinary tensor product.

To describe morphisms in $\HBS_2$, one chooses a resolution of each Bott--Samelson bimodule by graded free $R$-bimodules and considers morphisms in the homotopy category $\Kb(R\lh\gmod\lh R)$ between these resolutions.  The monoidal structure on morphisms in $\HBS_2$ then corresponds to the usual tensor product of morphisms between these resolutions.  The discussion of \S\ref{s:prelim} shows that $\HBS_2$ can be viewed either as a genuine monoidal category with two grading shifts $(1)$ and $\ll1\rr$, or as a $\Z^2$-graded super-monoidal category with parity function $\Z^2 \rightarrow \Z/2$ sending $(j,k)\mapsto k \mod 2$.

\subsection{Diagrammatic presentation}
\label{ss:hsbim-diag}

A diagrammatic monoidal presentation for $\BS_2$ (and more generally for the $\GL$-realization of any $S_n$) was given by Elias--Khovanov \cite{ek}: they defined a $\bk$-linear strict monoidal category $\Diag_2$ by generators and relations, together with a monoidal equivalence
\begin{equation} \label{eq:EK-equiv}
 \cF: \Diag_2 \simto \BS_2.
\end{equation}

Subsequent work of Elias \cite{elias} and Elias--Williamson \cite{ew} defined the \emph{diagrammatic Hecke category} for more general Coxeter systems and realizations. This presentation of the Hecke category was crucial in \cite{amrwb}, which constructs the monoidal Koszul duality functor by generators and relations.

We follow the same approach for the Ext-enhancement. For this, we need a presentation for $\HBS_2$. The category we construct will be a $\Z^2$-graded super-monoidal category with respect to the parity function $p:\Z^2 \rightarrow \Z/2$ sending $(j,k)\mapsto k$ mod 2.

\begin{defn}
The \emph{Ext-enhanced $\GL_2$ diagrammatic Hecke category} $\HDiag_2$ is the $\bk$-linear $\Z^2$-graded strict super-monoidal category defined by the diagrammatic presentation below.
\end{defn}

The objects of $\HDiag_2$ are the same as those of $\Diag_2$; they are indexed by expressions, and the object corresponding to $\uw$ is denoted by $B_\uw$:
\[
 B_\varnothing, \quad B_s, \quad B_{(s,s)}, \quad B_{(s,s,s)}, \quad \ldots.
\]

As in $\Diag_2$, a morphism $B_\uw \to B_\uv$ in $\HDiag_2$ is a $\bk$-linear combination of diagrams in a planar strip, where each diagram has bottom boundary $\uw$, top boundary $\uv$, and is made up of local pieces given by a list of generating morphisms.  

Each generating morphism in $\Diag_2$ of degree $m$ is also a generating morphism in $\HDiag_2$ of bidegree $(m,0)$:
\begin{subequations}
\begin{equation} \label{eqn:ogens}
{\arraycolsep=9pt
\begin{array}{ccccc}
\text{generator:}
&
\begin{array}{c}\begin{tikzpicture}[scale=0.5,thick,baseline]
 \draw (0,-1) to (0,0); \node[dot] at (0,0) {};
\end{tikzpicture}\end{array}
&
\begin{array}{c}\begin{tikzpicture}[scale=-0.5,thick,baseline]
 \draw (0,-1) to (0,0); \node[dot] at (0,0) {};
\end{tikzpicture}\end{array}
&
\begin{array}{c}\begin{tikzpicture}[scale=0.5,thick,baseline]
 \draw (-1,-1) -- (0,0); \draw (1,-1) -- (0,0) -- (0,1);
\end{tikzpicture}\end{array}
&
\begin{array}{c}\begin{tikzpicture}[scale=-0.5,thick,baseline]
 \draw (-1,-1) -- (0,0); \draw (1,-1) -- (0,0) -- (0,1);
\end{tikzpicture}\end{array} \\[8pt]
\text{bidegree:} & (1,0) & (1,0) & (-1,0) & (-1,0)
\end{array}
}
\end{equation}
In the graphical calculus of $\Diag_2$ one also has generating morphisms consisting of homogeneous elements of $R$ floating in an empty region.  In $\HDiag_2$ we allow elements of the algebra of self $\Ext$'s of $R$ to float around.  We now describe this algebra explicitly.

Recall that $R = \k[x_1,x_2]$.  Then the bigraded algebra of self $\Ext$'s of $R$ is canonically isomorphic to $\HR := R\otimes \HLambdav$, where $\HLambdav=\Lambda(V\ll 1\rr)$. We write $\xi_1, \xi_2, \xi_s$ for the elements of $V\ll 1\rr$ corresponding to $y_1, y_2, \alpha_s^\vee \in V$, so that $\HLambdav$ can be identified with the exterior algebra in the two generators $\xi_1,\xi_2$ of degrees $(0,1)$.  In other words
\[
\HR = R\otimes_\k \HLambdav \cong \k[x_1,x_2]\otimes \Lambda[\xi_1,\xi_2],\qquad \deg x_i =(2,0),\qquad\deg \xi_i =(0,1).
\]
In addition to the generators \eqref{eqn:ogens}, $\HDiag_2$ has the following generators:
\begin{equation} \label{eqn:hgens}
{\arraycolsep=15pt
\begin{array}{ccc}
\text{generator:}
&
\begin{array}{c}\begin{tikzpicture}[scale=0.4,thick,baseline]
 \draw (0,-1) to (0,1); \node[hdot] at (0,0) {};
\end{tikzpicture}\end{array}
&
\begin{array}{c}\begin{tikzpicture}[scale=0.5,thick,baseline]
 \draw (-0.5,-0.5) rectangle (0.5,0.5); \node at (0,0) {$f$};
\end{tikzpicture}\end{array} \\[8pt]
\text{bidegree:} & (-2,1) & \deg f
\end{array}
}
\end{equation}
Here, $f$ is a homogeneous element in $\HR$, and $\deg f$ denotes its bidegree.  We also define the shorthands
\begin{equation} \label{eqn:hupdot-hdowndot-def}
\begin{array}{c}\begin{tikzpicture}[xscale=0.3,yscale=0.3,thick,baseline]
 \draw (0,-1) -- (0,1); \node[hdot] at (0,1) {};
\end{tikzpicture}\end{array}
:=
\begin{array}{c}\begin{tikzpicture}[xscale=0.3,yscale=0.3,thick,baseline]
 \draw (0,-1) -- (0,1); \node[hdot] at (0,0) {}; \node[dot] at (0,1) {};
\end{tikzpicture}\end{array}\ \ ,
\qquad \qquad
\begin{array}{c}\begin{tikzpicture}[xscale=0.3,yscale=-0.3,thick,baseline]
 \draw (0,-1) -- (0,1); \node[hdot] at (0,1) {};
\end{tikzpicture}\end{array}
:=
\begin{array}{c}\begin{tikzpicture}[xscale=0.3,yscale=-0.3,thick,baseline]
 \draw (0,-1) -- (0,1); \node[hdot] at (0,0) {}; \node[dot] at (0,1) {};
\end{tikzpicture}\end{array}\ \ .
\end{equation}
\end{subequations}

The morphisms of $\HDiag_2$ satisfy the defining relations of $\Diag_2$, plus the following additional relations:
\begin{subequations}
\begin{gather}
\begin{array}{c}\begin{tikzpicture}[xscale=0.5,yscale=0.6,thick,baseline]
 \draw (-1,-1) -- (0,0) -- (0,1); \draw (1,-1) -- (0,0); \node[hdot] at (-0.5,-0.5) {};
\end{tikzpicture}\end{array}
=
\begin{array}{c}\begin{tikzpicture}[xscale=0.5,yscale=0.6,thick,baseline]
 \draw (-1,-1) -- (0,0) -- (0,1); \draw (1,-1) -- (0,0); \node[hdot] at (0,0.5) {};
\end{tikzpicture}\end{array}
=
\begin{array}{c}\begin{tikzpicture}[xscale=0.5,yscale=0.6,thick,baseline]
 \draw (-1,-1) -- (0,0) -- (0,1); \draw (1,-1) -- (0,0); \node[hdot] at (0.5,-0.5) {};
\end{tikzpicture}\end{array}, \qquad
\begin{array}{c}\begin{tikzpicture}[xscale=0.5,yscale=-0.6,thick,baseline]
 \draw (-1,-1) -- (0,0) -- (0,1); \draw (1,-1) -- (0,0); \node[hdot] at (-0.5,-0.5) {};
\end{tikzpicture}\end{array}
=
\begin{array}{c}\begin{tikzpicture}[xscale=0.5,yscale=-0.6,thick,baseline]
 \draw (-1,-1) -- (0,0) -- (0,1); \draw (1,-1) -- (0,0); \node[hdot] at (0,0.5) {};
\end{tikzpicture}\end{array}
=
\begin{array}{c}\begin{tikzpicture}[xscale=0.5,yscale=-0.6,thick,baseline]
 \draw (-1,-1) -- (0,0) -- (0,1); \draw (1,-1) -- (0,0); \node[hdot] at (0.5,-0.5) {};
\end{tikzpicture}\end{array}, \label{eqn:hdot-trivalent} \\
\begin{array}{c}\begin{tikzpicture}[xscale=0.3,yscale=0.4,thick,baseline]
 \draw (0,-1) -- (0,1); \node[dot] at (0,-1) {}; \node[hdot] at (0,0) {}; \node[dot] at (0,1) {};
\end{tikzpicture}\end{array}
=
\begin{array}{c}\begin{tikzpicture}[xscale=0.3,yscale=0.25,thick,baseline]
 \draw (0,-1) -- (0,1); \node[dot] at (0,-1) {}; \node[hdot] at (0,1) {};
\end{tikzpicture}\end{array}
=
\begin{array}{c}\begin{tikzpicture}[xscale=0.3,yscale=0.25,thick,baseline]
 \draw (0,-1) -- (0,1); \node[dot] at (0,1) {}; \node[hdot] at (0,-1) {};
\end{tikzpicture}\end{array}
=
\begin{array}{c}\begin{tikzpicture}[scale=0.25,thick,baseline]
 \draw (-1,-1) rectangle (1,1); \node at (0,0) {$\xi_s$};
\end{tikzpicture}\end{array}, \label{eqn:hbarbell} \\
\begin{array}{c}\begin{tikzpicture}[xscale=0.3,yscale=0.3,thick,baseline]
 \draw (0,-1.5) -- (0,1.5); \node[hdot] at (0,-0.5) {}; \node[hdot] at (0,0.5) {};
\end{tikzpicture}\end{array}
=
0, \label{eqn:hdot-square} \\
\begin{array}{c}\begin{tikzpicture}[scale=0.4,thick,baseline]
 \draw (-0.5,-0.5) rectangle (0.5,0.5); \node at (0,0) {$f$};
\end{tikzpicture}\end{array}
+
\begin{array}{c}\begin{tikzpicture}[scale=0.4,thick,baseline]
 \draw (-0.5,-0.5) rectangle (0.5,0.5); \node at (0,0) {$f'$};
\end{tikzpicture}\end{array}
=
\begin{array}{c}\begin{tikzpicture}[scale=0.4,thick,baseline]
 \draw (-1.2,-0.5) rectangle (1.2,0.5); \node at (0,0) {$f + f'$};
\end{tikzpicture}\end{array}, \qquad
\begin{array}{c}\begin{tikzpicture}[scale=0.4,thick,baseline]
 \draw (-0.5,0.5) rectangle (0.5,1.5); \node at (0,1) {$f$};
 \draw (-0.5,-1.5) rectangle (0.5,-0.5); \node at (0,-1) {$f'$};
\end{tikzpicture}\end{array}
=
\begin{array}{c}\begin{tikzpicture}[scale=0.4,thick,baseline]
 \draw (-1.2,-0.5) rectangle (1.2,0.5); \node at (0,0) {$f \cdot f'$};
\end{tikzpicture}\end{array}
\qquad \text{ for } f, f' \in \HR, \label{eqn:exterior-boxes-mult} \\
\begin{array}{c}\begin{tikzpicture}[scale=0.4,thick,baseline]
 \draw (0,-1.3) -- (0,1.3);
 \draw (0.5,-0.5) rectangle (1.5,0.5); \node at (1,0) {$\xi$};
\end{tikzpicture}\end{array}
=
\begin{array}{c}\begin{tikzpicture}[scale=0.4,thick,baseline]
 \draw (0,-1.3) -- (0,1.3);
 \draw (-2.3,-0.5) rectangle (-0.5,0.5); \node at (-1.4,0) {$s(\xi)$};
\end{tikzpicture}\end{array}
+
\alpha_s(\xi)
\begin{array}{c}\begin{tikzpicture}[scale=0.4,thick,baseline]
 \draw (0,0.5) -- (0,1.3); \node [hdot] at (0,0.5) {};
 \draw (0,-1.3) -- (0,-0.5); \node[dot] at (0,-0.5) {};
\end{tikzpicture}\end{array}
\qquad \text{ for } \xi \in V\ll-1\rr \subset \HLambdav. \label{eqn:exterior-forcing}
\end{gather}
\end{subequations}
This concludes the definition of $\HDiag_2$.

We will sometimes use $\star$ rather than juxtaposition for the monoidal product in $\HDiag_2$. We also identify elements of $\HR = R \otimes_\bk \HLambdav$ with $\bk$-linear combinations of products of the corresponding boxes.

It is a straightforward exercise in diagrammatics to derive the following further relations from the ones above:
\begin{subequations}
\begin{gather}
 \usebox\updothdowndot = \usebox\hupdotdowndot \ , \label{eqn:updot-downdot} \\
 \xi_s \star \usebox\downdot = \alpha_s \star \usebox\hdowndot \ , \label{eqn:downdot-hdowndot} \\
 \xi_s \star \usebox\updot = \alpha_s \star \usebox\hupdot \ , \label{eqn:updot-hupdot} \\
 \xi_s \star \usebox\hdot = \usebox\hdot \star \xi_s = 0.
\end{gather}
\end{subequations}
Below, the bigraded morphism spaces in $\HDiag_2$ will be denoted by $\gHom(-,-)$, with $\gEnd(B):=\gHom(B,B)$.  The following computation of morphism spaces is again straightforward, using the known (thanks to the equivalence \eqref{eq:EK-equiv}) morphism spaces in $\Diag_2$:
\begin{subequations} \label{eqn:hdiag-hom}
\begin{gather}
 \gEnd(B_\varnothing) = \HR = R \otimes_\bk \HLambdav, \\
 \gHom(B_\varnothing, B_s) = \left( \HR \star \usebox\updot + \HR \star \usebox\hupdot \right)/(\xi_s \star \usebox\updot - \alpha_s \star \usebox\hupdot), \\
 \gHom(B_s, B_\varnothing) = \left(\HR \star \usebox\downdot + \HR \star \usebox\hdowndot \right)/(\xi_s \star \usebox\downdot - \alpha_s \star \usebox\hdowndot), \\
 \gEnd(B_s) = \left(\HR \star \usebox\linemor + \HR \star \usebox\updotdowndot + \HR \star \usebox\hdot + \HR \star \usebox\updothdowndot \right)/(\xi_s \star \usebox\hdot \ , \xi_s \star \usebox\updotdowndot - \alpha_s \star \usebox\updothdowndot \,).
\end{gather}
\end{subequations}

In the rest of this paper, we work with $\HDiag_2$ rather than $\HBS_2$.

\begin{rmk} \label{rmk:hsbim-diag}
 One can show that there exists a $\bk$-linear monoidal equivalence $\HcF: \HDiag_2 \simto \HBS_2$ extending the equivalence \eqref{eq:EK-equiv}. Full details will appear in future work in a more general setting, but most of the work for this $\GL_2$ case is already contained in \cite[\S3.5]{gh} (where $\HDiag_2$ is denoted by $\mathcal{D}_2$). In the notation of \cite[\S3.5]{gh}, $\HcF$ is defined by generators and relations by sending
\[
 \usebox\updot \mapsto b,
 \qquad
 \usebox\downdot \mapsto b^*,
 \qquad
 \begin{array}{c}\begin{tikzpicture}[scale=0.2,baseline=-3pt] \draw (-1,-1) rectangle (1,1); \node at (0,0) {$x_i$}; \end{tikzpicture}\end{array} \mapsto x_i,
 \qquad
 \usebox\hdot \mapsto -\iota_{\varphi_2},
 \qquad
 \begin{array}{c}\begin{tikzpicture}[scale=0.2,baseline=-3pt] \draw (-1,-1) rectangle (1,1); \node at (0,0) {$\xi_i$}; \end{tikzpicture}\end{array} \mapsto \theta_i,
\]
which imply $\usebox\hupdot \mapsto \omega$, $\usebox\hdowndot \mapsto \omega^*$. Then \cite[Proposition~3.30]{gh} verifies many of the necessary relations among these images, and also computes many of the hom spaces needed to show fully faithfulness.
\end{rmk}

\section{Ext-enhanced free-monodromic tilting sheaves}
\label{sec:hfmtilt}
Starting from the diagrammatic Hecke category $\Diag_2$, \cite{amrwa} defined a category $\FM_2$ of ``free-monodromic complexes'' and a full subcategory $\Tilt_2$ of ``free-monodromic tilting sheaves.'' This construction can be repeated with $\Diag_2$ replaced by $\HDiag_2$ to yield categories $\FM_2^\hoch$ of \emph{Ext-enhanced free-monodromic complexes} and $\Tilt_2^\hoch$ of \emph{Ext-enhanced free-monodromic tilting sheaves}.

We assume familiarity with \cite{amrwa}, and only recall those details of these constructions that are relevant for the calculations that follow. See, however, the remark at the end of \S\ref{ss:diagram-seq} for differences in grading and sign convention.

\subsection{Diagram sequences} \label{ss:diagram-seq}
We need to clarify how the extra grading in $\HDiag_2$ is treated when dealing with complexes and, soon, free-monodromic complexes.

Let $(\HDiag_2)^\oplus$ denote the additive envelope of $\HDiag_2$, obtained by formally adjoining finite direct sums.  Let $\DS' := (\HDiag_2)^{\oplus,(1),\ll 1 \rr}$ denote the envelope of $\HDiag_2$ in which we adjoin not just formal direct sums, but also formal grading shifts $B(m)\ll n\rr$ with $B\in (\HDiag_2)^\oplus$ and $m,n\in \Z$.  The morphism spaces in this envelope are by definition the bigraded hom spaces
\[
\Hom^{j,k}_{\DS'}(B(m)\ll n\rr, B'(m')\ll n'\rr) :=\gHom^{j+m'-m,k+n'-n}(B,B').
\]
By construction, $\DS'$ is a $\Z^2$-graded super-monoidal category with parity form
\[
\Z^2\times \Z^2\rightarrow \Z,\qquad \qquad ((j,k),(j',k'))\mapsto kk'.
\]
We want to consider formal complexes in which the ``chain groups'' are objects in $\DS'$.  As a preliminary, we first construct the category which later will play the role of ``complexes with zero differential.''

Let $\Seq(\DS')$ be the category of $\Z$-graded sequences of objects in $\DS'$.  An object of $\Seq(\DS')$ is a $\Z$-indexed sequence $\cF = (\cF^i)_{i \in \Z}$ with $\cF^i\in \DS'$. The morphism spaces in $\Seq(\DS')$ are the $\Z^3$-graded $\k$-modules $\Hom_{\Seq(\DS')}^{\Z\times \Z\times\Z}(\cF,\cG)$ with
\[
\Hom_{\Seq(\DS')}^{i,j,k}(\cF,\cG):= \prod_{p\in \Z}\Hom_{\DS'}^{j,k}(\cF^p,\cG^{p+i}).
\]
Morphisms of degree $(i,j,k)$ are said to have \emph{cohomological degree} $i$, \emph{Soergel degree} $j$, and \emph{Hochschild degree} $k$. 

Let $\Seqb(\DS')\subset \Seq(\DS')$ denote the full subcategory consisting of finite sequences $\cF$, for which $\cF^i=0$ for all but finitely many $i\in \Z$.  The category $\Seqb(\DS')$ inherits an operation $\star$ defined on objects by
\[
(\cF\star \cG)^p :=\bigoplus_{q+q'=p} \cF^q\star \cG^{q'}.
\]
If $f\in \Hom_{\Seq(\DS')}^{i,j,k}(X,X')$ and $g\in \Hom_{\Seq(\DS')}^{i',j',k'}(Y,Y')$ are two homogeneous morphisms, then we define a morphism $f\star g\in \Hom_{\Seq(\DS')}^{i+i',j+j',k+k'}(X\star Y,X'\star Y')$ by its restrictions
\[
(f\star g)|_{X^p\star Y^q} := (-1)^{pi'} f|_{X^p} \star g|_{Y^q} \in \Hom_{\DS'}^{j+j',k+k'}(X^p\star Y^q, X^{p+i}\star Y^{q+i'}).
\]

It is an exercise to show that this gives $\Seqb(\DS')$ the structure of a $\Z^3$-graded super-monoidal category with the associated parity form
\[
\ip{\ ,\ }:\Z^3\times\Z^3\rightarrow \Z/2,\qquad\qquad \ip{(i,j,k),(i',j',k')}:= ii'+kk',
\]
In other words, the composition of morphisms satisfies the super-exchange law
\[
(f'\star g')\circ (f\star g) = (-1)^{\ip{\deg g', \deg f}}(f'\circ f)\star (g'\circ g).
\]

\begin{rmk}
It is important to note that if $f$ and $g$ are morphisms in $\Seqb(\DS')$ with $\deg f=(2i,j,k)$ and $\deg g=(i',j',2k')$, then $(f \star \id) \circ (\id \star g) =  (\id \star g) \circ (f \star \id)$ with no sign.  In this way the parities associated to the cohomological and Hochschild degrees are independent.
\end{rmk}

Let $[1]:\Seq(\DS')\rightarrow \Seq(\DS')$ denote the downward grading shift functor $\cF[1]^p= \cF^{p+1}$.  On morphisms $[1]$ acts by a conventional sign:
\[
f[1] := (-1)^{|f|} f,
\]
where $|f|\in \Z$ denotes the first component of $\deg(f)\in \Z^3$. The sign here guarantees that the functor $\cF\mapsto \one[1]\star \cF$ is naturally isomorphic to $\cF\mapsto \cF[1]$.


Altogether, $\Seq(\DS')$ is equipped with three grading shift functors $[i],(j),\ll k \rr$, defined on objects by
\[
\cF[l](m)\ll n \rr^p =\cF^{p+l}(m)\ll n\rr,
\]
and satisfying
\[
\Hom_{\Seq(\DS')}^{i,j,k}(\cF,\cG[l](m)\ll n \rr) = \Hom_{\Seq(\DS')}^{i+l,j+m,k+n}(\cF,\cG).
\]
We will also use the combined shift
\[
\la1\ra := [1](-1).
\]

\begin{rmk} The Hochschild degree $0$ part of our construction recovers the categories of \cite{amrwa}, with two differences in grading and sign convention. First, degree $(i,j)$ in \cite{amrwa} corresponds to degree $(i-j,j,0)$ in this paper. Second, \cite{amrwa} defined an operation $\ustar$ on $\Seqb((\Diag_2)^{\oplus,(1)})$ by a careful choice of signs so that $\Seqb((\Diag_2)^{\oplus,(1)})$ became $\Z^2$-graded super-monoidal with parity function $(i,j) \mapsto i$. One could just as well have used the induced product $\star$ as above and parity function $(i,j) \mapsto i+j$ (corresponding to $(i,j,0) \mapsto i$ in our convention).
\end{rmk}

\subsection{Free-monodromic complexes}
Below we will consider differential $\Z^3$-graded categories with differentials of degree $(1,0,0)$ in which the Leibniz rule takes the form
\[
d(f\circ g) = d(f)\circ g + (-1)^{|f|}f\circ d(g),
\]
where $|f|$ is the cohomological degree of $f$.  Such categories will be called \emph{dggg categories}.  For instance $\Seq(\DS')$ defined in \S\ref{ss:diagram-seq} is a dggg category with zero differential.

\begin{defn} Let $\HBEdg:=\Chb(\DS')$ denote the dggg category of formal finite complexes over $\DS'$.  Objects of this category are pairs $(\cF,\d)$ where $\cF$ is an object of $\Seqb(\DS')$ and $\d\in \End^{1,0,0}_{\Seq(\DS')}(\cF)$ satisfies $\d \circ \d =0$.  The morphism spaces in $\HBEdg$ are the complexes
\[
\uHom_{\HBE}^{\Z\times \Z\times \Z}(\cF,\cG):=\Hom_{\Seq(\DS')}^{\Z\times \Z\times \Z}(\cF,\cG)
\]
with differential
\[
f\mapsto \d_\cG\circ f -(-1)^{|f|}f\circ \d_\cF.
\]
Let $\HBE$ denote the cohomology category of $\HBEdg$; it has the same objects, but morphism spaces are the degree $(0,0,0)$ chain maps modulo homotopy.
\end{defn}

We define some $\Z^3$-graded $\bk$-algebras in preparation for our definition of $\HFMdg$. Regard $R = \Sym(V^*(-2))$ and $\HLambdav = \Lambda(V\ll 1\rr)$ now as being $\Z^3$-graded, concentrated in cohomological degree $0$. Also define
\[
 \Lambda := \Sym(V^*[1](-2)), \qquad R^\vee := \Sym(V\la-2\ra).
\]
We write $\nu_1, \nu_2, \nu_s$ for the elements of $V^*[1](-2)$ corresponding to $x_1, x_2, \alpha_s \in V^*$. Thus
\[
R=\k[x_1,x_2],\qquad \HLambdav = \Lambda[\xi_1, \xi_2], \qquad \Lambda = \Lambda[\nu_1, \nu_2], \qquad R^\vee = \bk[y_1, y_2],
\]
with degrees
\[
\deg x_i = (0,2,0), \quad \deg \xi_i = (0,0,1), \quad \deg \nu_i = (-1,2,0), \quad \deg y_i = (2,-2,0).
\]
Consider the differential $\Z^3$-graded algebra
\[
K:=\Lambda \otimes_\k R^\vee\otimes_\k R = \k[x_1,x_2,y_1,y_2]\otimes_\k\Lambda[\nu_1,\nu_2]
\]
with differential $\kappa$ determined by $\kappa(x_i)=0=\kappa(y_i)$ and $\kappa(\nu_i)=x_i$ together with the Leibniz rule with respect to the cohomological degree. Let $K\otimes_R \Seq(\DS')$ denote the category with the same objects as $\Seq(\DS')$, but morphism spaces given by
\[
\Hom_{K\otimes_R \Seq(\DS')}^{\Z\times \Z\times \Z}(\cF,\cG):= K\otimes_R \Hom^{\Z\times\Z\times \Z}_{\Seq(\DS')}(\cF,\cG)
\]
with composition
\[
(a\otimes f)\circ (b\otimes g) = (-1)^{|f||b|}ab\otimes (f\circ g).
\]
Hom spaces in $K\otimes_R \Seq(\DS')$ inherit a differential from $K$, which we will continue to denote by $\kappa$.  Thus, $K\otimes_R \Seq(\DS')$ is a dggg category.

\begin{defn}
For each $\cF\in \Seq(\DS')$, let $\Theta_{\cF}\in \End_{K\otimes_R \Seq(\DS')}^{2,0,0}(\cF)$ denote the closed endomorphism
\[
\Theta_\cF := \sum_i y_i\otimes(\id_\cF\star x_i).
\]

Let $\HFMdg$ denote the dggg category whose objects are pairs $(\cF,\d)$ where $\cF\in \Seq(\DS')$ and $\d\in \End_{K\otimes_R \Seq(\DS')}^{1,0,0}(\cF)$ is an element such that
\[
\kappa(\d)+\d \circ \d = \Theta_\cF.
\]
The hom spaces in $\HFMdg$ are by definition the complexes
\[
\uHom_{\HFM}^{\Z\times \Z\times \Z}((\cF,\d),(\cG,\d')):=\Hom_{K \otimes_R \Seq(\DS')}^{\Z\times \Z\times \Z}(\cF,\cG)
\]
with differential
\[
f\mapsto \kappa(f) + \d'\circ f - (-1)^{|f|}f\circ \d.
\]
Let $\HFM$ denote the cohomology category of $\HFMdg$.  Objects of $\HFMdg$ or $\HFM$ are called \emph{($\Ext$-enhanced) free-monodromic complexes}.
\end{defn}

We adopt the usual terminology of dg categories.  A homogeneous morphism $f\in \uHom_{\HFM}^{i,j,k}((\cF,\d),(\cG,\d'))$ is \emph{closed} if $\kappa(f) + \d'\circ f - (-1)^i f\circ \d =0$ and \emph{exact} if $f = \kappa(h) + \d'\circ h + (-1)^i h\circ \d$ for some $h\in \uHom^{i-1,j,k}_{\HFM}((\cF,\d),(\cG,\d'))$.  Closed morphisms of degree zero are called \emph{chain maps}, and exact morphisms of degree zero are called nullhomotopic chain maps.

Morphisms in $\HFM$ are degree zero chain maps modulo homotopy.
\begin{rmk}
The endomorphisms $\Theta_\cF$ define a closed degree $(2,0,0)$ element $\Theta$ of the center of $K\otimes_R \Seq(\DS').$  The centrality of $\Theta$ is used in the proof that the definition above actually defines a dggg category.
\end{rmk}

Let us recall the free-monodromic complexes $\Tmon_\varnothing$ and $\Tmon_s$ defined in \cite[\S5.3.1]{amrwa} and \cite[\S5.3.2]{amrwa}, respectively. Define the following elements of $K$:
\begin{equation} \label{eqn:theta-w-defn}
 \theta = \theta_\id := \sum_{i = 1}^2 \nu_iy_i, \qquad \theta_s := \sum_{i = 1}^2 s(\nu_i)y_i.
\end{equation}
The following equations are easily checked:
\begin{equation} \label{eqn:theta-minus-theta-s-quadratic-odd-rels}
 \theta - \theta_s = \nu_s \alpha_s^\vee, \qquad
 (y_1+y_2)\nu_1\nu_2 = \nu_s \theta_s = \nu_s \theta = - \theta_s\nu_s = - \theta\nu_s.
\end{equation}

Both $\theta$ and $\theta_s$ are degree $(1,0,0)$ elements of the (graded) center: for $w \in \{\id, s\}$,
\begin{equation} \label{eqn:theta-w-graded-central}
 \theta_w \circ f = (-1)^{|f|}f \circ \theta_w.
\end{equation}

The underlying sequence of $\Tmon_\varnothing$ consists of $B_\varnothing$ in position $0$, and $\delta_{\Tmon_\varnothing} = \theta$.

The underlying sequence of $\Tmon_s$ is $(\ldots, 0, B_\varnothing(-1), B_s, B_\varnothing (1), 0, \ldots)$, where the non-zero terms are in positions $-1$ through $1$.  This sequence can also be denoted by $B_\varnothing(-1)[1]\oplus B_s[0]\oplus B_\varnothing(1)[-1]$, and
{ \small
\[
\Tmon_s = \left(B_\varnothing(-1)[1]\oplus B_s[0]\oplus B_\varnothing(1)[-1]\ \ , \ \  \  \d\ := \ \sqmatrix{
\theta\otimes \id & 0 & 0 \\
1\otimes \usebox\downdot & \ \ \ \theta_s\otimes \id \ \ \ & \a_s^\vee\otimes \usebox\downdot \\
-\nu_s\otimes \id & 1\otimes \usebox\updot & \theta_s\otimes \id} \right).
\]
}

These objects can be depicted by the following pictures:
\[
\Tmon_\varnothing =
\begin{tikzcd}
B_\varnothing, \ar[loop above, "\theta"]
\end{tikzcd}
\qquad \qquad
\Tmon_s =
\begin{tikzcd}
B_{\varnothing}(1) \ar[loop, in=0, out=20, distance=20, "\theta_s\otimes\id" pos=0.51] \ar[d, bend left=80, "\ \a_s^\vee \otimes \usebox\downdot" pos=0.38] \\
B_s \ar[u, "\usebox\updot"] \ar[loop, in=-30, out=-10, distance=30, "\theta_s\otimes \id" pos=0.5] \\
B_{\varnothing}(-1). \ar[u, "\usebox\downdot"] \ar[uu, bend left=60, "-\nu_s \otimes \id" near end] \ar[loop right, distance=20, "\theta\otimes\id"]
\end{tikzcd}
\]

Much of \cite{amrwa} was guided by the dream that the category of free-monodromic complexes (or at least a large subcategory thereof) should also be monoidal. To this end, \cite[\S6]{amrwa} defined an operation $\hatstar$ (``free-monodromic convolution'') for a certain class of free-monodromic complexes called ``convolutive'' and morphisms between them. For an expression $\uw = (s, \ldots, s)$, let
\[
 \Tmon_\uw := \Tmon_s \hatstar \cdots \hatstar \Tmon_s,
\]
and let $\Tilt_2$ be the full subcategory of $\FM_2$ consisting of $\Tmon_\uw\la n \ra$ for expressions $\uw$ and $n \in \Z$. As a particular case of the main result \cite[Theorem~11.4.2]{amrwa}, $\Tilt_2$ admits a monoidal structure with the operation $\hatstar$ and identity $\Tmon_\varnothing$. The hardest part of this result was to show that $\hatstar$ is a bifunctor, i.e.~that morphisms in $\Tilt_2$ satisfy the exchange law
\begin{equation} \label{eq:tilt-interchange}
(f'\hatstar g')\circ (f\hatstar g) = (f'\circ f)\hatstar (g'\circ g).
\end{equation}

In $\FM_2^\hoch$, the operation $\hatstar$ can be defined by the same formula (but replacing $\ustar$ by $\star$, see the Remark at the end of \S\ref{ss:diagram-seq}) for a similarly defined class of convolutive complexes and morphisms between them. Let $\Tilt_2^\hoch$ be the full subcategory of $\FM_2^\hoch$ consisting of $\Tmon_\uw\la n \ra\ll m \rr$ for expressions $\uw$ and $n, m \in \Z$.

\begin{conj} \label{conj:htilt-monoidal}
 $(\Tilt_2^\hoch, \hatstar, \Tmon_\varnothing)$ admits a monoidal structure extending that on $(\Tilt_2, \hatstar, \Tmon_\varnothing)$.
\end{conj}
This again reduces to the exchange law \eqref{eq:tilt-interchange}, but even with the various simplifications possible in this $\GL_2$ case, it does not seem clear how to adapt the proof in \cite{amrwa}. We will return to this conjecture in future work.

In this paper, we content ourselves with defining an Ext-enhanced monoidal Koszul duality functor $\HPhi: \HDiag_2 \to \Tilt_2^\hoch$ (see Theorem~\ref{thm:hmkd}) assuming Conjecture~\ref{conj:htilt-monoidal}.

\subsection{A canonical free-monodromic morphism}
\label{ss:fm-phi}
In preparation for Theorem~\ref{thm:hmkd}, we now define and study an endomorphism $\hphi_s$ of $\Tmon_s$ that will be the image under $\HPhi$ of the degree $(-2,1)$ endomorphism of $B_s$ introduced in \eqref{eqn:hgens}.

Recall that $\Tmon_s$ has underlying diagram sequence $B_\varnothing(-1)[1]\ \oplus \ B_s (0)[0]\ \oplus \ B_\varnothing(1)[-1]$. Below, we define morphisms involving $\Tmon_s$ via matrices. Define the endomorphism
\begin{equation} \label{eqn:hphi-def}
\hphi_s \ := \ \sqmatrix{
0 &\ \  -\nu_s \otimes \usebox\hupdot\ \  & -1\otimes \xi_s\\
0 &  0 &  \nu_s \otimes \usebox\hdowndot\\
0 & 0 & 0 \\
} \ \in \ \uEnd_{\HFM}^{-2,2,1}(\Tmon_s),
\end{equation}
which can be depicted as follows:
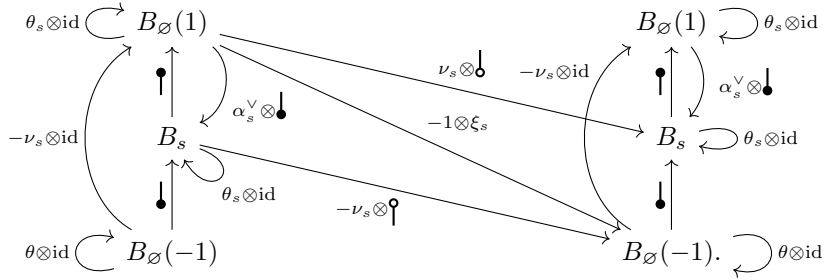
\[
\begin{tikzcd}[row sep=large, column sep=large]
B_\varnothing(1)
\ar[loop left, distance=20, "\theta_s\otimes \id"]
\ar[d, bend left=60,"\alpha_s^\vee \otimes \usebox\downdot" pos=0.46]
\ar[rrrrdd, "-1 \otimes \xi_s"]
\ar[rrrrd, "\nu_s \otimes \usebox\hdowndot"]
&&&& B_\varnothing(1)
\ar[loop right, distance=20, "\theta_s\otimes \id"]
\ar[d, bend left=40, "\ \alpha_s^\vee\otimes \usebox\downdot "]
\\
B_s \ar[u, "\usebox\updot"]
\ar[loop right, in=-60, out=-20, distance=20, "\theta_s\otimes\id" pos=0.6]
\ar[rrrrd, "-\nu_s \otimes \usebox\hupdot"']
&&&&B_s
\ar[loop right, distance=20, "\theta_s\otimes \id"]
\ar[u, "\usebox\updot"]
\\
B_\varnothing(-1)
\ar[u, "\usebox\downdot"]
\ar[uu, bend left=60, "-\nu_s \otimes \id"]
\ar[loop left, distance=20, "\theta\otimes \id"]
&&&& B_\varnothing(-1).
\ar[loop right, distance=20, "\theta\otimes \id"]
\ar[u, "\usebox\downdot"]
\arrow[bend left=60]{uu}[near end]{-\nu_s \otimes \id}
\end{tikzcd}
\]
We will check that $\hphi_s$ is closed in the course of the proof of the following lemma, which is the main goal of this subsection.
\begin{lem} \label{lem:fm-phi-one-dimensional}
 We have
 \[
  \Hom_{\FM_2^\hoch}(\Tmon_s, \Tmon_s\la-2\ra\ll1\rr) = \bk \cdot \hphi_s.
 \]
\end{lem}
The proof of Lemma~\ref{lem:fm-phi-one-dimensional} occupies the rest of this subsection.

Let $f \in \uEnd_{\HFM}^{-2,2,1}(\Tmon_s)$. We claim that for degree reasons, $f$ must be of the form
\[
f = \sqmatrix{
0 & r_3 \otimes \usebox\hupdot & 1\otimes \xi\\
0 & \ \ \ r_2 \otimes \usebox\hdot \ \ \  & r_1 \otimes \usebox\hdowndot\\
0 & 0 & 0
}
\]
for some $r_1, r_2, r_3 \in \Lambda[\nu_1, \nu_2]$, where $r_1, r_3$ are linear and $r_2$ is quadratic, and some linear $\xi \in V\ll-1\rr \subset \HLambdav$.

Let us comment briefly on this claim. For instance, consider the component $f_{31}$ of $f$.  This component lives in the hom space
\begin{multline*}
f_{31} \in \uHom_{\HFM}^{-2,2,1}(B_\varnothing(-1)[1], B_\varnothing(1)[-1]) 
\cong \uHom_{\HFM}^{-2,2,1}(B_\varnothing, B_\varnothing(2)[-2]) \\
\cong \uHom_{\HFM}^{-4,4,1}(B_\varnothing, B_\varnothing)
\cong (\k[x_1,x_2,y_1,y_2]\otimes_\k \Lambda[\nu_1,\nu_2,\xi_1,\xi_2])^{-4,4,1}.
\end{multline*}
To obtain the degree $(-4,4,1)$ we try to solve the equation
\[
a\deg x_i + b\deg y_i + c\deg \nu_i + e\deg \xi_i = (-4,4,1),\qquad\qquad a,b,c,e\in \Z_{\geq 0}.
\]
Given that $\deg x_i = (0,2,0)$, $\deg y_i = (2,-2,0)$, $\deg \nu_i = (-1,2,0)$, and $\deg \xi_i = (0,0,1)$, we see that $2b-c=-4$, which forces $c\geq 4$.  This implies that $f_{31}=0$ since any degree 3 and higher expression in the odd variables $\nu_1,\nu_2$ is zero.

The rest of the claim above may be checked by repeating a similar degree argument for each component using \eqref{eqn:hdiag-hom}.

We consider the equation for $f$ to be closed:
\begin{equation}\label{eqn:phi-s-candidate-chain-map}
\kappa(f) + \d_{\Tmon_s}\circ f - f \circ \d_{\Tmon_s} = 0.
\end{equation}
Compute:
{\small
\begin{eqnarray*}
    \d_{\Tmon_s}\circ f
&=&
    \sqmatrix{
    \theta\otimes \id & 0 & 0 \\
    1\otimes \usebox\downdot & \ \ \ \theta_s\otimes \id \ \ \ & \a_s^\vee\otimes \usebox\downdot \\
    -\nu_s\otimes \id & 1\otimes \usebox\updot & \theta_s\otimes \id
    }
    \circ 
    \sqmatrix{
    0 & r_3 \otimes \usebox\hupdot & 1\otimes \xi\\
    0 & \ \ \ r_2 \otimes \usebox\hdot \ \ \  & r_1 \otimes \usebox\hdowndot\\
    0 & 0 & 0
    } \\
&=&
\sqmatrix{
            \ 0  \ \ \ 
            &
            \theta r_3\otimes \usebox\hupdot
            &
            \ \ \ \theta\otimes \xi
            \\
            \ 0  \ \ \
            &
            -r_3\otimes \usebox\hupdotdowndot + \theta_s r_2\otimes \usebox\hdot
            &
            \ \ \ 1\otimes (\usebox\downdot\star\xi) +\theta_s r_1\otimes \usebox \hdowndot
            \\
            \ 0 \ \ \ 
            &
            -\nu_s r_3 \otimes \usebox\hupdot + r_2\otimes \usebox\ohupdot
            &
            \ \ \ - \nu_s\otimes \xi-r_1\otimes \usebox\hobarbell
            }
\end{eqnarray*}}
and
{\small
\begin{eqnarray*}
    f\circ \d_{\Tmon_s}
&=&
    \sqmatrix{
    0 & r_3 \otimes \usebox\hupdot & 1\otimes \xi\\
    0 & \ \ \ r_2 \otimes \usebox\hdot \ \ \  & r_1 \otimes \usebox\hdowndot\\
    0 & 0 & 0
    }
    \circ
    \sqmatrix{
    \theta\otimes \id & 0 & 0 \\
    1\otimes \usebox\downdot & \ \ \ \theta_s\otimes \id \ \ \ & \a_s^\vee\otimes \usebox\downdot \\
    -\nu_s\otimes \id & 1\otimes \usebox\updot & \theta_s\otimes \id
    }\\
&=&
    \sqmatrix{
        r_3\otimes \usebox\ohbarbell - \nu_s\otimes \xi
    &
        -r_3\theta_s \otimes \usebox\hupdot + 1\otimes (\xi\star \usebox\updot)
    &
        r_3 \alpha_s^\vee \otimes \usebox\ohbarbell + \theta_s\otimes \xi
    \\
        r_2\otimes \usebox\ohdowndot + r_1\nu_s \otimes \usebox\hdowndot
    &
       \ \  \ \ \ \ r_2\theta_s\otimes \usebox\hdot
    + r_1\otimes \usebox\updothdowndot \ \ \ \  \ \ 
    &
    r_2\a_s^\vee \otimes \usebox\ohdowndot - r_1\theta_s\otimes \usebox\hdowndot
    \\
    0 & 0 & 0
    }.
\end{eqnarray*}}

To obtain all the correct signs above, recall that the composition of morphisms in $\HFMdg$ satisfies $(r\otimes g)\circ (r'\otimes g') = (-1)^{|g||r'|} (rr')\otimes (g\circ g')$.  As a useful rule of thumb, remember that each component of $\d$ (resp.~$f$) has odd (resp.~even) cohomological degree.  Consider for example the $(1,2)$ component $f_{12}=r_3\otimes\usebox\hupdot$.  Since  $r_3$ is linear in $\nu_1,\nu_2$, it has odd cohomological degree.  Thus, in this context $\usebox\hupdot$ should also be regarded as being odd, since $f_{12}$ is even.

Now, viewing \eqref{eqn:phi-s-candidate-chain-map} as a matrix equation and taking the $(1,1)$, $(3,3)$ components (where $\kappa(f)$ is zero) yields
\begin{gather*}
0 = -r_3\otimes \usebox\ohbarbell + \nu_s\otimes \xi \overset{\eqref{eqn:hbarbell}}{=} -r_3 \otimes \xi_s + \nu_s \otimes \xi, \\
0 = -\nu_s\otimes \xi - r_1\otimes \usebox\hobarbell \overset{\eqref{eqn:hbarbell}}{=} -r_1 \otimes \xi_s - \nu_s \otimes \xi.
\end{gather*}
Thus we must have $\xi = c\xi_s$ and $r_1 = -r_3 = -c\nu_s$ for some $c \in \bk$.  For these choices for $r_1,r_3$, the $(2,1)$ component of \eqref{eqn:phi-s-candidate-chain-map} yields
\[
0 = -r_2 \otimes \usebox\ohdowndot + 0 \overset{\eqref{eqn:hupdot-hdowndot-def}}{=} -r_2 \otimes \usebox\hdowndot
\]
and $\usebox\hdowndot \neq 0$ by \eqref{eqn:hdiag-hom}, so $r_2 = 0$. Thus $f = -c \hphi_s$.

It remains to show that $\hphi_s$ is closed. We have
\[
\kappa(\hphi_s) = \sqmatrix{
0 & \ \ \ -1 \otimes (\a_s\star \usebox\hupdot) \ \ \ & 0\\
0 & 0  & 1 \otimes (\a_s\star \usebox\hdowndot)\\
0 & 0 & 0
}.
\]
Combined with the computations above, we obtain that $\kappa(\hphi_s)+\d_{\Tmon_s}\circ \hphi_s - \hphi_s\circ \d_{\Tmon_s}$ equals
{ \footnotesize
\[
\sqmatrix{
    0
&
   \ \ \ \ \ \  -\a_s\star\usebox\hupdot - \theta \nu_s\otimes \usebox\hupdot
    - \nu_s\theta_s \otimes \usebox\hupdot + \xi_s\star\usebox\updot \ \ \ \ \ \ 
&
   -\theta\otimes \xi_s + \nu_s \alpha_s^\vee \otimes \usebox\ohbarbell +\theta_s\otimes \xi_s 
\\
    0
&
  \nu_s\otimes \usebox\hupdotdowndot -\nu_s\otimes \usebox\updothdowndot
&
     \a_s\star\usebox\hdowndot - \usebox\downdot\star\xi_s +\theta_s\nu_s\otimes \usebox\hdowndot +\nu_s\theta_s\otimes \usebox\hdowndot
\\
0
&
\nu_s \nu_s \otimes \usebox\hupdot
&
0
}.
\]
}
The components in positions $(1,1)$, $(2,1)$, and $(3,3)$ are zero by the discussion above.  It is an exercise to show that the remaining entries are also zero using \eqref{eqn:theta-minus-theta-s-quadratic-odd-rels} and the defining relations in $\HDiag_2$.

Finally, degree considerations entirely similarly to the ones earlier show that
\[
 \uEnd_\HFM^{-3,2,1}(\Tmon_s) = 0,
\]
so $\hphi_s$ is not nullhomotopic. This completes the proof of Lemma~\ref{lem:fm-phi-one-dimensional}.

\subsection{Free-monodromic ``Hochschild unit and counit'' morphisms}
\label{ss:fm-hunit}
As in \cite[\S5.3.4]{amrwa}, define morphisms $\hat\eta_s \in \uHom_{\HFM}^{1,-1,0}(\Tmon_\varnothing,\Tmon_s)$ and $\hat\epsilon_s \in \uHom_{\HFM}^{1,-1,0}(\Tmon_s, \Tmon_\varnothing)$ by the matrices
\begin{equation} \label{eqn:heta-heps-def}
\hat\eta_s \ \ := \ \ 
\sqmatrix{-\a_s^\vee\otimes\id \\ 0\\ 1\otimes \id},
\qquad\qquad
\hat\epsilon_s \ \ :=  \ \ \sqmatrix{  -1\otimes \id \ \  & 0 & \ \ 0} .
\end{equation}
These morphisms can be depicted by the following pictures:
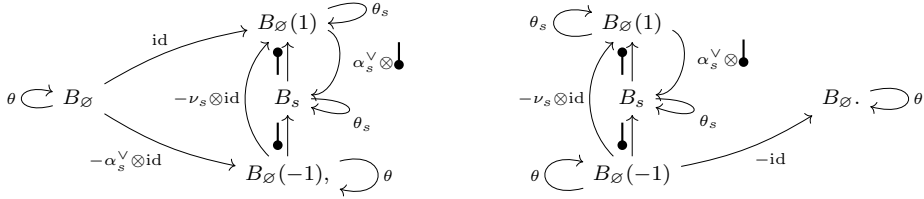
\footnotesize{
\[
\begin{tikzcd}[column sep=50pt]
&B_\varnothing(1)
\ar[loop, in=0, out=20, distance=20, "\theta_s" pos=0.51]
\ar[d, bend left=80, "\ \a_s^\vee \otimes \usebox\downdot\ \  ", pos=.38] \\
B_\varnothing \ar[loop left,"\theta"]
  \ar[ur, bend left=10, "\id"]
  \ar[dr, bend right=10, "-\a_s^\vee \otimes \id"']
&B_s \ar[u, "\usebox\updot"] \ar[loop right, in=-20, out=0, distance=20, "\theta_s" pos=0.6] \\
&B_\varnothing(-1), \ar[u, "\usebox\downdot"] \ar[uu, bend left=45, "-\nu_s \otimes \id"] \ar[loop right, distance=20, "\theta"]
\end{tikzcd}
\qquad\qquad
\begin{tikzcd}[column sep=50pt]
B_\varnothing(1) \ar[loop left, distance=20, "\theta_s"] \ar[d, bend left=80, "\  \a_s^\vee \otimes \usebox\downdot" pos=0.35] \\
B_s \ar[u, "\usebox\updot"] \ar[loop right, in=-20, out=0, distance=20, "\theta_s"]
& B_\varnothing.\ar[loop right, distance=20, "\theta"] \\
B_\varnothing(-1) \ar[u, "\usebox\downdot"] \ar[uu, bend left=45, "-\nu_s \otimes \id"] \ar[loop left, distance=20, "\theta"]
\ar[ur, bend right=10, "-\id"']
\end{tikzcd}
\]
}
\normalsize

\begin{rmk}
It is an exercise (see also \cite[\S5.3.4]{amrwa}) to show that these morphisms are closed; they can be viewed as degree zero chain maps $\Tmon_\varnothing \ip{-1}\rightarrow \Tmon_s$ and $\Tmon_s\rightarrow \Tmon_\varnothing\ip{1}$.
\end{rmk}

Now, define the (closed) morphisms
\begin{equation} \label{eqn:fm-hdot-defn}
 \hheta_s = \hphi_s \circ \hat\eta_s \in \uHom_{\HFM}^{-1,1,1}(\Tmon_\varnothing, \Tmon_s), \qquad\qquad \hheps_s = \hat\epsilon_s \circ \hphi_s \in \uHom_{\HFM}^{-1,1,1}(\Tmon_s,\Tmon_\varnothing),
\end{equation}
given in terms of components by
\[
\hheta_s = \sqmatrix{-1\otimes \xi_s \\ \nu_s \otimes \usebox\hdowndot \\ 0},
\qquad\qquad
\hheps_s = \sqmatrix{0 &\ \ \nu_s\otimes \usebox\hupdot  \ \ & 1\otimes \xi_s},
\]
which can be depicted as follows:
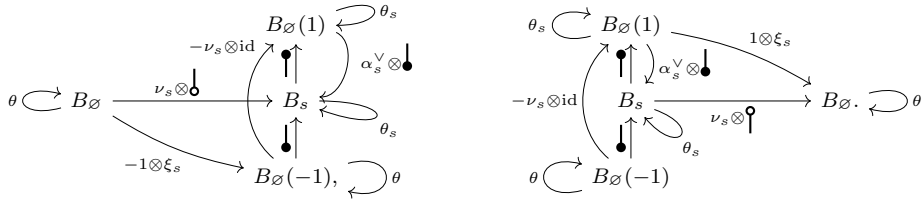
\footnotesize{
\[
\begin{tikzcd}[column sep=50pt]
& B_\varnothing(1)
\ar[loop, in=0, out=20, distance=20, "\theta_s" pos=0.51]
\ar[d, bend left=80, "\ \alpha_s^\vee \otimes \usebox\downdot" pos=0.37] \\
B_\varnothing
\ar[loop left, distance=20, "\theta"]
\ar[r, "\nu_s \otimes \usebox\hdowndot" pos=0.4]
\ar[rd, bend right=10, "-1 \otimes \xi_s"' pos=0.6] & B_s
\ar[u, "\usebox\updot"] \ar[loop, in=-20, out=0, distance=30, "\theta_s" pos=0.6] \\
& B_\varnothing(-1),
\ar[u, "\usebox\downdot"]
\arrow[bend left=50]{uu}[pos=0.85]{-\nu_s \otimes \id}
\ar[loop right, distance=20, "\theta"]
\end{tikzcd}
\qquad\qquad
\begin{tikzcd}[column sep=50pt]
B_\varnothing(1)
\ar[loop left, distance=20, "\theta_s"]
\ar[d, bend left=40, "\alpha_s^\vee \otimes \usebox\downdot" pos=0.48]
\ar[rd, bend left=10, "1 \otimes \xi_s"] \\
B_s \ar[u, "\usebox\updot"]
\ar[loop, in=-50, out=-20, distance=20, "\theta_s" pos=0.6]
\ar[r,  "\nu_s \otimes \usebox\hupdot"'] & B_\varnothing.
\ar[loop right, distance=20, "\theta"]\\
B_\varnothing(-1) \ar[u, "\usebox\downdot"]
\ar[uu, bend left=50, "-\nu_s \otimes \id"]
\ar[loop left, distance=20, "\theta"]
\end{tikzcd}
\]
}
\normalsize

The following is the free-monodromic analogue of \eqref{eqn:exterior-forcing}.
\begin{lem} \label{lem:fm-exterior-forcing}
 Let $\xi \in V\ll-1\rr \subset \HLambdav$. We have
 \begin{equation} \label{eqn:fm-exterior-forcing}
  \id_{\Tmon_s} \hatstar \xi = s(\xi) \hatstar \id_{\Tmon_s} + \alpha_s(\xi)\hheta_s \circ \hat\epsilon_s
 \end{equation}
as morphisms $\Tmon_s \to \Tmon_s\ll1\rr$ in $\FM_2^\hoch$.
\end{lem}
\begin{proof}
According to the definition of $\hatstar$, we have $\id_{\Tmon_s} \hatstar \xi = \id_{\Tmon_s} \star \xi$ and $s(\xi) \hatstar \id_{\Tmon_s} = s(\xi) \star \id_{\Tmon_s}$ since these morphisms do not involve nontrivial $\Lambda$ or $R^\vee$ components. Using this, the equality $\xi - s(\xi) = \alpha_s(\xi)\xi_s$, and the explicit chain maps representing $\hheta_s$ and $\hat\epsilon_s$ described above, one computes that
\[
 \id_{\Tmon_s} \hatstar \xi - s(\xi) \hatstar \id_{\Tmon_s} - \alpha_s(\xi)\hheta_s \circ \hat\epsilon_s
\]
may be represented as
\begin{multline*}
\sqmatrix{\xi -s(\xi)&0&0\\
0 & \ \ \ \id_{B_s}\star \xi - s(\xi)\star\id_{B_s}\ \ \ & 0 \\
0&0&\xi -s(\xi)}
- \a_s(\xi)\sqmatrix{
\xi_s&0&0\\
-\nu_s\otimes\usebox\hdowndot &0&0\\
0&\ \ \ \ 0\ \ \ \ &0} \\
\overset{\eqref{eqn:exterior-forcing}}{=} \a_s(\xi)\sqmatrix{
0&0&0\\
\nu_s\otimes\usebox\hdowndot & \ \ \ \ \usebox\updothdowndot \ \ \ \ &0 \\
0&0& \xi_s}.
\end{multline*}

This chain map is nullhomotopic with nullhomotopy given by $\a_s(\xi)h$, where $h =\smMatrix{0&0&0\\ 0&\ \ 0\ \ & \usebox\hdowndot \\ 0&0&0}$.  Indeed, $\kappa(h)=0$ and
{\footnotesize
\begin{multline*}
\d_{\Tmon_s}\circ h + h\circ \d_{\Tmon_s} =
\sqmatrix{\theta & \ \ 0 \ \ & 0 \\ \usebox\downdot & \theta_s & \a_s^\vee\otimes\usebox\downdot \\ -\nu_s & \usebox\updot & \theta_s}
\circ
\sqmatrix{0&0&0\\ 0 &\ \ 0\ \ & \usebox\hdowndot \\ 0&0&0} \\
+
\sqmatrix{0&0&0\\ 0&\ \ 0\ \ & \usebox\hdowndot \\ 0&0&0}
\circ
\sqmatrix{\theta & \ \ 0 \ \ & 0 \\ \usebox\downdot & \theta_s & \a_s^\vee\otimes\usebox\downdot \\ -\nu_s & \usebox\updot & \theta_s}
\\
=
\sqmatrix{0&\ \ \ 0\ \ \ \ &0\\ 0 & 0 & \theta_s\otimes\usebox\hdowndot \\ 0 &0 & \usebox\hobarbell}
+
\sqmatrix{0&\ \ \ 0\ \ \ \ &0\\ \nu_s\otimes \usebox\hdowndot & \usebox\updothdowndot
& -\theta_s\otimes \usebox\hdowndot \\ 0 & 0& 0 }
 \overset{\eqref{eqn:hbarbell}}{=} 
\sqmatrix{0&\ \ \ 0\ \ \ \ &0\\ \nu_s\otimes \usebox\hdowndot & \usebox\updothdowndot
& 0 \\ 0 & 0& \xi_s}.
\end{multline*}}
\end{proof}

\section{Ext-enhanced monoidal Koszul duality}
\label{sec:hmkd}

Let $\varphi: V^* \simto V$ be the isomorphism $\varphi(x_i) = y_i$. Then $\varphi$ identifies the $\GL_2$ realization of $S_2$ with its dual (in particular, $\varphi$ is $S_2$-equivariant and $\varphi(\alpha_s) = \alpha_s^\vee$), hence induces a monoidal equivalence between $\Diag_2$ and the diagrammatic Hecke category associated to the dual realization. Composing this with the monoidal Koszul duality functor $\Phi$ of \cite[Theorem~4.1]{amrwb}, we obtain a $\bk$-linear monoidal equivalence
\[
 \Phi_\sym: (\Diag_2, \star, B_\varnothing) \simto (\Tilt_2, \hatstar, \Tmon_\varnothing)
\]
satisfying $\Phi_\sym \circ (1) = \la1\ra \circ \Phi_\sym$. (Here, $\sym$ stands for ``self-dual.'') Extend $\varphi$ multiplicatively to a $\bk$-algebra isomorphism $\varphi: R \simto R^\vee$. The functor $\Phi_\sym$ is defined by generators and relations: on objects by $\Phi_\sym(B_s) = \Tmon_s$, and on morphisms by
\begin{subequations} \label{eq:Phi-defn}
\begin{gather}
 \Phi_\sym\left(
  \begin{tikzpicture}[thick,scale=0.07,baseline=-2pt]
   \node at (0,0) {$x$};
   \draw (-3,-3) rectangle (3,3);
  \end{tikzpicture}
 \right)
 := \mu_{\Tmon_\varnothing}(\varphi(x)) : \Tmon_\varnothing \to \Tmon_\varnothing\la2m\ra, \quad \text{where } x \in R^{0,2m,0}, \\
 \Phi_\sym\left(
  \begin{tikzpicture}[thick,scale=0.07,baseline]
   \draw (0,5) to (0,0);
   \node at (0,0) {$\bullet$};
  \end{tikzpicture}
 \right) := \hat\eta_s: \Tmon_\varnothing\la -1\ra \to \Tmon_s,
 \qquad
 \Phi_\sym\left(
  \begin{tikzpicture}[thick,scale=0.07,baseline=-5pt]
   \draw (0,-5) to (0,0);
   \node at (0,0) {$\bullet$};
  \end{tikzpicture}
 \right) := \hat\epsilon_s: \Tmon_s \to \Tmon_\varnothing\la 1\ra, \\
 \Phi_\sym\left(
  \begin{tikzpicture}[thick,baseline=-2pt,scale=0.07]
   \draw (-4,5) to (0,0) to (4,5);
   \draw (0,-5) to (0,0);
  \end{tikzpicture}
  \right) := \hat{b}_1: \Tmon_s \to \Tmon_s \hatstar \Tmon_s\la-1\ra,
  \qquad
 \Phi_\sym\left(
  \begin{tikzpicture}[thick,baseline=-2pt,scale=-0.07]
   \draw (-4,5) to (0,0) to (4,5);
   \draw (0,-5) to (0,0);
  \end{tikzpicture}
  \right) := \hat{b}_2: \Tmon_s \hatstar \Tmon_s \to \Tmon_s\la-1\ra.
\end{gather}
\end{subequations}
Here, $\mu_{\Tmon_\varnothing}$ is the left monodromy action defined in \cite[Theorem~5.2.2]{amrwa}. The morphisms $\hat\eta_s, \hat\epsilon_s$ were defined explicitly (see \eqref{eqn:heta-heps-def}), whereas $\hat{b}_1, \hat{b}_2$ were defined more indirectly in \cite[\S4.2.3]{amrwb}. See \cite[\S4.2]{amrwb} for more details.

\begin{thm} \label{thm:hmkd}
If Conjecture~\ref{conj:htilt-monoidal} holds, then there exists a $\bk$-linear monoidal functor
\[
 \HPhi: (\HDiag_2, \star, B_\varnothing) \to (\Tilt_2^\hoch, \hatstar, \Tmon_\varnothing)
\]
extending $\Phi$. In particular, $\Phi(B_\uw) = \Tmon_\uw$ for any expression $\uw$, and $\HPhi \circ (1) = \langle 1 \rangle \circ \HPhi$. Moreover, $\HPhi \circ \ll1\rr = \ll1\rr \circ \HPhi$.
\end{thm}
In fact, we expect $\HPhi$ to be an equivalence.  One reason to believe in this Koszul duality relating Ext-enhanced categories $\HDiag_2$ and $\Tilt_2^\hoch$ is that it is closely related to---and in a certain sense would explain---the ``mirror symmetry'' of triply-graded Khovanov--Rozansky homology of knots \cite{dgr}, which has been verified in all available computations \cite{eh, hog2, Mellit, HogMellit}.


The proof of Theorem~\ref{thm:hmkd} occupies the rest of this paper. As in \cite{amrwb}, we construct $\HPhi$ by generators and relations: we specify the images of the generating morphisms of $\HDiag_2$, and check that these images satisfy the defining relations of $\HDiag_2$.

The images of the generating morphisms of $\Diag_2$ are the same as for $\Phi_\sym$, as given in \eqref{eq:Phi-defn}. For the new generators \eqref{eqn:hgens}, define
\begin{subequations}
\begin{gather}
 \HPhi\left(
  \begin{tikzpicture}[thick,scale=0.07,baseline=-2pt]
   \node at (0,0) {$\xi$};
   \draw (-3,-3) rectangle (3,3);
  \end{tikzpicture}
 \right)
 :=
 1 \otimes
 \begin{tikzpicture}[thick,scale=0.07,baseline=-2pt]
  \node at (0,0) {$\xi$};
   \draw (-3,-3) rectangle (3,3);
 \end{tikzpicture}
 \otimes 1: \Tmon_\varnothing \to \Tmon_\varnothing\ll m\rr, \quad \text{where } \xi \in (\HLambdav)^{0,0,m}, \\
 \HPhi\left(
  \begin{tikzpicture}[thick,scale=0.07,baseline=-3pt]
   \node at (-0.5,0) {$\phantom{.}$};
   \draw (0,-5) to (0,5); \node[hdot] at (0,0) {};
   \node at (0.5,0) {$\phantom{.}$};
  \end{tikzpicture}
 \right)
 :=
 \hphi_s: \Tmon_s \to \Tmon_s\la-2\ra\ll1\rr.
\end{gather}
\end{subequations}
Then $\HPhi(\usebox\hdowndot) = \hheta_s$ and $\HPhi(\usebox\hupdot) = \hheps_s$ by \eqref{eqn:hupdot-hdowndot-def} and \eqref{eqn:fm-hdot-defn}.

Let us verify the defining relations of $\HDiag_2$ for these images. By \cite[Theorem~4.1]{amrwb}, the relations of $\Diag_2$ hold. In particular, we will use below the Frobenius unit relations
\begin{equation} \label{eqn:fm-frob-unit-rels}
  (\id_{\Tmon_s} \hatstar \hat{\epsilon}_s) \circ  \hat{b}_1  = \id_{\Tmon_s}, \qquad \hat{b}_2 \circ ( \id_{\Tmon_s} \hatstar \hat{\eta}_s) = \id_{\Tmon_s}.
\end{equation}

It remains to verify the new relations from \S\ref{ss:hsbim-diag}. The relations \eqref{eqn:hbarbell} and \eqref{eqn:hdot-square}, which say $\hat\epsilon_s \circ \hphi_s \circ \hat\eta_s = \xi_s$ and $\hphi_s \circ \hphi_s = 0$, follow by direct computation via the explicit chain maps for $\hphi_s, \hat\epsilon_s, \hat\eta_s$. The relation \eqref{eqn:exterior-boxes-mult} is clear, and \eqref{eqn:exterior-forcing} was verified in \eqref{eqn:fm-exterior-forcing}.

Only the relations \eqref{eqn:hdot-trivalent} remain. First, let us show that
\begin{equation} \label{eqn:fm-hdot-pop}
  \hat{b}_2 \circ (\hheta_s \hatstar \id_{\Tmon_s}) = \hphi_s = \hat{b}_2 \circ (\id_{\Tmon_s} \hatstar \hheta_s).
\end{equation}
(Here and in what follows, we omit shifts on morphisms from the notation.) We prove the first equality; the second equality is similar. By Lemma~\ref{lem:fm-phi-one-dimensional}, $\hat{b}_2 \circ (\hheta_s \hatstar \id_{\Tmon_s}) = c\hphi_s$ for some $c \in \bk$. Then
\[
 \hheta_s \overset{\eqref{eqn:fm-frob-unit-rels}}{=} \hat{b}_2 \circ (\id_{\Tmon_s} \hatstar \hat\eta_s) \circ (\hheta_s \hatstar \id_{\Tmon_\varnothing}) = \hat{b}_2 \circ (\hheta_s \hatstar \id_{\Tmon_s}) \circ (\id_{\Tmon_\varnothing} \hatstar \hat\eta_s) = c\hphi_s \circ \hat\eta_s \overset{\eqref{eqn:fm-hdot-defn}}{=} c\hheta_s
\]
and $\hheta_s \neq 0$, so $c = 1$, which proves \eqref{eqn:fm-hdot-pop}.

Diagrammatically, \eqref{eqn:fm-hdot-pop} says that the images of $\HPhi$ satisfy
\begin{equation} \label{eqn:hdot-pop}
\begin{array}{c}\begin{tikzpicture}[xscale=0.3,yscale=0.4,thick,baseline]
 \draw (-1,-1) -- (0,0); \draw (2,-2) -- (0,0) -- (0,1); \node[hdot] at (-1,-1) {};
\end{tikzpicture}\end{array}
=
\begin{array}{c}\begin{tikzpicture}[xscale=0.3,yscale=0.4,thick,baseline]
 \draw (0,-1.5) -- (0,1.5); \node[hdot] at (0,0) {};
\end{tikzpicture}\end{array}
=
\begin{array}{c}\begin{tikzpicture}[xscale=-0.3,yscale=0.4,thick,baseline]
 \draw (-1,-1) -- (0,0); \draw (2,-2) -- (0,0) -- (0,1); \node[hdot] at (-1,-1) {};
\end{tikzpicture}\end{array} \ \ , \qquad
\begin{array}{c}\begin{tikzpicture}[xscale=0.3,yscale=-0.4,thick,baseline]
 \draw (-1,-1) -- (0,0); \draw (2,-2) -- (0,0) -- (0,1); \node[hdot] at (-1,-1) {};
\end{tikzpicture}\end{array}
=
\begin{array}{c}\begin{tikzpicture}[xscale=0.3,yscale=-0.4,thick,baseline]
 \draw (0,-1.5) -- (0,1.5); \node[hdot] at (0,0) {};
\end{tikzpicture}\end{array}
=
\begin{array}{c}\begin{tikzpicture}[xscale=-0.3,yscale=-0.4,thick,baseline]
 \draw (-1,-1) -- (0,0); \draw (2,-2) -- (0,0) -- (0,1); \node[hdot] at (-1,-1) {};
\end{tikzpicture}\end{array} \ \ .
\end{equation}
Now, \eqref{eqn:hdot-trivalent} can be deduced from \eqref{eqn:hdot-pop} and Frobenius associativity in $\Diag_2$:
\begin{gather*}
\begin{array}{c}\begin{tikzpicture}[xscale=0.25,yscale=0.3,thick,baseline]
 \draw (-2,-2) -- (0,0) -- (0,2); \draw (2,-2) -- (0,0); \node[hdot] at (-1,-1) {};
\end{tikzpicture}\end{array}
\overset{\eqref{eqn:hdot-pop}}{=}
\begin{array}{c}\begin{tikzpicture}[xscale=0.25,yscale=0.3,thick,baseline]
 \draw (-1.5,-1.5) -- (0,0) -- (0,2); \draw (0.5,-2) -- (-0.75,-0.75); \draw (2,-2) -- (0,0); \node[hdot] at (-1.5,-1.5) {};
\end{tikzpicture}\end{array}
=
\begin{array}{c}\begin{tikzpicture}[xscale=0.25,yscale=0.3,thick,baseline]
 \draw (-1.5,-1.5) -- (0,0) -- (0,2); \draw (-0.5,-2) -- (0.75,-0.75); \draw (2,-2) -- (0,0); \node[hdot] at (-1.5,-1.5) {};
\end{tikzpicture}\end{array}
\overset{\eqref{eqn:hdot-pop}}{=}
\begin{array}{c}\begin{tikzpicture}[xscale=0.25,yscale=0.3,thick,baseline]
 \draw (-2,-2) -- (0,0) -- (0,2); \draw (2,-2) -- (0,0); \node[hdot] at (0,1) {};
\end{tikzpicture}\end{array}, \\
\begin{array}{c}\begin{tikzpicture}[xscale=0.25,yscale=0.3,thick,baseline]
 \draw (-2,2) -- (0,0) -- (0,-2); \draw (2,2) -- (0,0); \node[hdot] at (-1,1) {};
\end{tikzpicture}\end{array}
\overset{\eqref{eqn:hdot-pop}}{=}
\begin{array}{c}\begin{tikzpicture}[xscale=0.25,yscale=0.3,thick,baseline]
 \draw (-2,2) -- (0,0) -- (0,-2); \draw (-1,1) -- (-1.5,0.5); \draw (2,2) -- (0,0); \node[hdot] at (-1.5,0.5) {};
\end{tikzpicture}\end{array}
=
\begin{array}{c}\begin{tikzpicture}[xscale=0.25,yscale=0.3,thick,baseline]
 \draw (-2,2) -- (0,0) -- (0,-2); \draw (-1,1) -- (-1.5,0.5) -- (-1.5,-0.5); \draw (2,2) -- (0,0); \node[hdot] at (-1.5,-0.5) {};
\end{tikzpicture}\end{array}
=
\begin{array}{c}\begin{tikzpicture}[xscale=0.25,yscale=0.3,thick,baseline]
 \draw (-2,2) -- (0,0.5) -- (0,-0.5) -- (-1,-1.5); \draw (0,-0.5) -- (2,-2); \draw (2,2) -- (0,0.5); \node[hdot] at (-1,-1.5) {};
\end{tikzpicture}\end{array}
\overset{\eqref{eqn:hdot-pop}}{=}
\begin{array}{c}\begin{tikzpicture}[xscale=0.25,yscale=0.3,thick,baseline]
 \draw (-2,2) -- (0,0) -- (0,-2); \draw (2,2) -- (0,0); \node[hdot] at (0,-1) {};
\end{tikzpicture}\end{array},
\end{gather*}
and similarly for the vertical reflections.

We have thus shown that $\HPhi$ is a well-defined monoidal functor. It is clear from the construction that $\HPhi$ commutes with $\ll1\rr$. This completes the proof of Theorem~\ref{thm:hmkd}.


\end{document}